\documentclass[11pt,letterpaper]{amsart}
\usepackage[letterpaper,hmargin=1.25in,vmargin=1.5in]{geometry}
\usepackage{latexsym, amssymb, amsmath}
\usepackage{booktabs} 
\usepackage{multirow}
\usepackage{xcolor}

\newtheorem{theorem}{Theorem}[section]
\newtheorem{lemma}[theorem]{Lemma}

\newtheorem{proposition}[theorem]{Proposition}
\newtheorem{remark}[theorem]{Remark}
\theoremstyle{definition}
\newtheorem{definition}[theorem]{Definition}
\newtheorem{example}{Example}

\makeatletter
    
    \@addtoreset{equation}{section}
  \makeatother

\begin{document}

\title[$\lambda$-biharmonic Riemannian submersions]{$\lambda$-biharmonic Riemannian submersions from manifolds with constant sectional curvature}

\author{Shun Maeta}
\address{Department of Mathematics, Chiba University, 1-33, Yayoicho, Inage, Chiba, 263-8522, Japan.}
\curraddr{}
\email{shun.maeta@faculty.gs.chiba-u.jp~{\em or}~shun.maeta@gmail.com}

\author{Miho Shito}
\address{Department of Mathematics, Chiba University, 1-33, Yayoicho, Inage, Chiba, 263-8522, Japan.}
\email{shito.miho@gmail.com}
\thanks{The first author is partially supported by the Grant-in-Aid for Scientific Research (C), No.23K03107, Japan Society for the Promotion of Science.}

\date{}

\dedicatory{}


\commby{}

\begin{abstract}

In this paper, we study $\lambda$-biharmonic Riemannian submersions, which generalize biharmonic Riemannian submersions. 
We prove nonexistence results for $\lambda$-biharmonic Riemannian submersions from $(n+1)$-dimensional Riemannian manifolds with constant sectional curvature $c$ to $n$-dimensional Riemannian manifolds. 
Our results show that the critical value $\lambda=2(n-1)c$ plays a decisive role. 
We prove that if $c \ge 0$, or if $c<0$ and $\lambda \neq 2(n-1)c$, 
then any $\lambda$-biharmonic Riemannian submersion must be harmonic.
On the other hand, when $\lambda=2(n-1)c$ with $c<0$, 
we construct explicit examples of $\lambda$-biharmonic Riemannian submersions.
 \end{abstract}

\maketitle



\section{Introduction}\label{intro}

A harmonic map is a critical point of the energy functional
\[
E(\phi)=\frac12\int_M |d\phi|^2\,d\mu_g,
\]
where $\phi:(M,g)\to(N,h)$ is a smooth map. 
Equivalently, $\phi$ is harmonic if and only if its tension field
\[
\tau(\phi)=\operatorname{tr}_g(\nabla d\phi)
\]
vanishes identically \cite{EL1983,ES1964}. 
Harmonic maps provide a fundamental tool for studying the geometry of manifolds.

When harmonic maps do not exist under given geometric or topological constraints, 
it is natural to consider maps that minimize the failure of harmonicity \cite{LM2008}. 
A natural measure of this failure is the bienergy
\[
E_2(\phi)=\frac12\int_M |\tau(\phi)|^2\,d\mu_g,
\]
introduced by Eells and Lemaire \cite{EL1983}. 
Its critical points are called biharmonic maps.

Jiang \cite{J1986} derived the first variation formula for the bienergy and showed 
that $\phi$ is biharmonic if and only if $\tau_2(\phi)=0$, where
\begin{equation}\label{Jiang}
\tau_2(\phi)
=
\left(
\nabla^{\phi}_{e_i}\nabla^{\phi}_{e_i}
-
\nabla^{\phi}_{\nabla^M_{e_i} e_i}
\right)\tau(\phi)
-
R^N\bigl(d\phi(e_i),\tau(\phi)\bigr)d\phi(e_i).
\end{equation}
The vector field $\tau_2(\phi)$ is called the bitension field of $\phi$.
Here, $\nabla^{\phi}$ denotes the induced connection, and $R^N$ denotes the Riemannian curvature tensor of $N$, defined by
\[
R^N(X, Y)Z = \nabla^N_X \nabla^N_Y Z - \nabla^N_Y \nabla^N_X Z - \nabla^N_{[X,Y]} Z,
\]
for vector fields $X, Y, Z$ on $N$, and $\{e_i\}$  is a local orthonormal frame field of $M$.
Since every harmonic map is biharmonic, 
it is natural to study \textbf{proper biharmonic maps}, namely, biharmonic maps that are not harmonic.
Consequently, as represented by Chen's conjecture \cite{C1991} 
and the Balmu\c{s}--Montaldo--Oniciuc conjecture \cite{BMO2008}, 
much of the research has focused on the existence and nonexistence of proper biharmonic maps 
\cite{BMO2010,BLO2021,CMO2001,CMO2002,CI1998,D1998,FZ2021,FZ2023,FO,GLV2021,HV1995,LM2017,Maeta2014,Maeta2017,MOR2016,NU2011,CM2013,OT2012,OC2020}.
In our previous work \cite{MS2025}, we classified biharmonic Riemannian submersions from manifolds with constant sectional curvature.

Besides the unconstrained variational problem for the bienergy, 
one may also consider a constrained version involving the energy functional.
In 2008, Loubeau and Montaldo \cite{LM2008} introduced the concept of biminimal immersions 
by considering a constrained variational problem using a Lagrange multiplier $\lambda$.
In this paper, we adapt this idea to the setting of Riemannian submersions.

Riemannian submersions are a fundamental tool
for understanding the structure of a manifold 
by decomposing it into a lower-dimensional base space and fibers \cite{ON1966}.
While biharmonic immersions have been extensively investigated, 
their dual case, biharmonic Riemannian submersions, has remained comparatively less explored.
Nevertheless, this topic has attracted increasing attention in recent years
 \cite{O2002,AO2019,WO2025,WO2024,WO2023-1,WO2011,Ou2024}.
 An important early contribution is due to Oniciuc \cite{O2002}, who studied biharmonic Riemannian submersions under the assumption that the tension field is basic.
He proved nonexistence results under nonpositive Ricci curvature and provided examples 
when the tension field is a unit Killing vector field.
These nonexistence results were derived by means of a global Bochner--Weitzenb\"ock argument, 
in the spirit of Jiang's classical approach to biharmonic maps~\cite{J1986}.
Later, Wang and Ou classified biharmonic Riemannian submersions from 3-dimensional Riemannian manifolds with constant sectional curvature into a surface.~\cite{WO2011}
and extended the theory to 3-dimensional BCV spaces~\cite{WO2024} and product spaces $M^2 \times \mathbb{R}$~\cite{WO2025}.

In this paper, we consider Riemannian submersions
\[
\phi:(M^{n+1},g)\to (N^n,h),
\]
that is, the case of one-dimensional fibers.
Our main goal is to classify $\lambda$-biharmonic Riemannian submersions from manifolds with constant sectional curvature.
To this end, we define a \textbf{$\lambda$-biharmonic Riemannian submersion} as follows.
\begin{definition}[\cite{E1974}, \cite{LM2008}]

Let $\phi : (M^{m}, g) \to (N^{n}, h)$ be a Riemannian submersion between Riemannian manifolds.
If there exists a real constant $\lambda \in \mathbb{R}$ such that $\phi$ is a critical point of the $\lambda$-bienergy
\[
E_{2,\lambda}(\phi) := E_2(\phi) + \lambda E(\phi),
\]
then $\phi$ is called a $\lambda$-biharmonic Riemannian submersion.
In this case, $\phi$ satisfies
\[
\tau_{2,\lambda}(\phi) := \tau_2(\phi) - \lambda \tau(\phi) = 0,
\]
which follows from the Euler--Lagrange equations for harmonic and biharmonic maps.
The vector field $\tau_{2,\lambda}(\phi)$ is called the \textbf{$\lambda$-bitension field} of $\phi$.
\end{definition}
\begin{remark}
When $\lambda=0$, the $\lambda$-bienergy functional reduces to the classical bienergy $E_2(\phi)$.
Accordingly, a $0$-biharmonic Riemannian submersion coincides with the usual biharmonic Riemannian submersion.
Thus, the notion of a $\lambda$-biharmonic Riemannian submersion naturally generalizes the classical theory,
encompassing the case $\lambda=0$ and unifying the constrained and unconstrained approaches.
\end{remark}

\begin{definition}
A $\lambda$-biharmonic map $\phi$ is called \textbf{proper} if it is not harmonic,
that is,
\[
\tau_{2,\lambda}(\phi)=0
\quad \text{and} \quad
\tau(\phi)\neq 0.
\]
\end{definition}

Our results show that the critical value 
\[
\lambda = 2(n-1)c
\] 
plays a decisive role 
in the classification of $\lambda$-biharmonic Riemannian submersions 
from manifolds with constant sectional curvature $c$.
A classification of the existence and nonexistence of 
proper $\lambda$-biharmonic Riemannian submersions is presented in TABLE~\ref{class}.
Cases (a), (b), and (c) in TABLE~\ref{class} are obtained in \cite{MS2025}.
The main result is the following local classification theorem.

\begin{theorem}\label{thA}
For $n\geq2$, 
let $\phi : (M^{n+1}(c), g) \longrightarrow (N^n, h)$ be a Riemannian submersion from a Riemannian manifold with constant sectional curvature $c$. Assume that one of the following holds.
\begin{enumerate}
    \item $c\geq 0$,  
    \item $c<0$ and $\lambda \neq 2(n-1)c$.
\end{enumerate}

Then $\phi$ is $\lambda$-biharmonic if and only if it is harmonic.
\end{theorem}

Our proof is based on a reduction of the $\lambda$-biharmonic equation via integrability data 
and a detailed analysis of the resulting system under the constant sectional curvature assumption. 
In particular, our arguments do not rely on the Bochner--Weitzenb\"ock formula, which is global in nature.

For $c < 0$, at the critical value $\lambda = 2(n-1)c$,
Section 3 provides explicit examples of proper $\lambda$-biharmonic Riemannian submersions
\[
H^m(c)\longrightarrow H^n(c)
\]
for every $m > n$, where $H^k(c)$ denotes the $k$-dimensional hyperbolic space of constant curvature $c$.
In particular, these examples include the one-dimensional fiber case $m = n + 1$.

\begin{table}[htbp]
\centering
\caption{Classification of the existence/nonexistence of proper
$\lambda$-biharmonic Riemannian submersions with one-dimensional fibers
($n\ge2$) according to the curvature $c$ and the constant $\lambda$.}

\label{class}
\begin{tabular}{cccc}
\hline
Curvature $c$ 
& $\lambda=0$ (see \cite{MS2025})
& $\lambda=2(n-1)c$ 
& Other $\lambda$ \\ \hline
$c>0$ 
& (a) Does not exist 
& (d) Does not exist 
& (g) Does not exist  \\ 
$c=0$ 
& (b) Does not exist 
& 
\begin{tabular}{c}
(e) Does not exist\\
(same as (b))
\end{tabular}
& (h) Does not exist \\ 
$c<0$ 
& (c) Does not exist 
& (f) Exists in every dimension
& (i) Does not exist  \\ \hline
\end{tabular}
\end{table}

\section{Preliminaries}
In this section, we recall basic definitions and notations used throughout the paper. 

In 2011, Wang and Ou \cite{WO2011} introduced the integrability data 
$\{f_1, f_2, \kappa_1, \kappa_2, \sigma\}$ 
to analyze the structure of Riemannian submersions from 3-dimensional manifolds.
This approach provides a powerful framework for analyzing the biharmonic equation 
originally derived by Jiang \cite{J1986} (see \eqref{Jiang}),
enabling us to reduce the equation to a system of scalar functions on the manifold.

 Later, in 2019, Akyol and Ou \cite{AO2019} generalized these integrability data to the form $\{f_{ij}^k, \kappa_i, \sigma_{ij}\}$ for arbitrary dimensions.
The specific settings and definitions are as follows.

Let $\phi : (M^{n+1}, g) \to (N^n, h)$ be a Riemannian submersion.  
A local orthonormal frame adapted to $\phi$ consists of horizontal lifts of vector fields from the base manifold $N$, which locally span the horizontal distribution on $M$.  
Since basic vector fields locally span the horizontal distribution, such a frame can always be constructed (cf. page 2 in \cite{Mari}).

Let $\{e_1, \dots, e_n, e_{n+1}\}$ be an orthonormal frame adapted to the submersion $\phi$, where $e_1, \dots, e_n$ are horizontal lifts of the orthonormal frame $\{\varepsilon_1, \dots, \varepsilon_n\}$ on $N$, and $e_{n+1}$ is vertical along the fiber.
Since our argument is local, we fix such a frame on a sufficiently small open subset and suppress this restriction from the notation.

The Lie bracket $[e_i, e_{n+1}]$ is vertical and does not contribute to the horizontal component 
(Lemma 3, \cite{ON1966}).
Moreover, the bracket $[e_i, e_j]$ of horizontal lifts corresponds under $\phi$ to the bracket $[\varepsilon_i, \varepsilon_j]$ on $N$, preserving the Lie bracket structure
(Lemma 1, \cite{ON1966}).

If we assume that
\begin{equation}
[\varepsilon_i, \varepsilon_j] = F^k_{ij} \varepsilon_k, 
\end{equation}
where, in the sequel, $F^k_{ij} \in C^\infty(N)$ and the Einstein convention is used, then we have
\begin{equation}\label{data}
\begin{aligned}
&[e_i, e_{n+1}] = \kappa_i e_{n+1},  \\
&[e_i, e_j] = f^k_{ij} e_k - 2 \sigma_{ij} e_{n+1}, \quad i, j = 1, 2, \dots, n, 
\end{aligned}
\end{equation}
where $f^k_{ij} = F^k_{ij} \circ \phi$, $\kappa_i$, and  $\sigma_{ij} \in C^\infty(M)$ for all  $i, j = 1, 2, \dots, n$. We will call $\{f^k_{ij}, \kappa_i, \sigma_{ij}\}$ the  integrability data of the adapted frame of the Riemannian submersion $\phi$. It follows from (\ref{data}) that
\begin{equation}
f^k_{ij} = - f^k_{ji}, \quad \sigma_{ij} = - \sigma_{ji}, \quad ( i, j, k = 1, 2, \dots, n). 
\end{equation}

Using the integrability data, the $\lambda$-biharmonic condition can be rewritten as follows.

\begin{proposition}
Let $\phi : (M^{n+1}, g) \to (N^n, h)$ be a Riemannian submersion with the adapted frame $\{e_1, \dots, e_{n+1}\}$ and the integrability data  $\{f^k_{ij}, \kappa_i, \sigma_{ij}\}$.
Then, the Riemannian submersion $\phi$ is $\lambda$-biharmonic if and only if
\begin{equation}\label{eq:lambda}
\begin{aligned}
&-\sum_{i=1}^{n+1} e_i e_i(\kappa_k) + \sum_{i,j=1}^n P_{ii}^j e_j(\kappa_k) + \sum_{i=1}^n \kappa_i e_i(\kappa_k)\\
&\quad  - \sum_{i,j=1}^n \left[ 2 e_i(\kappa_j) P_{ij}^k + \kappa_j (e_i P_{ij}^k) + \kappa_j P_{ij}^l P_{il}^k - \kappa_i \kappa_j P_{ij}^k - \kappa_j P_{ii}^l P_{lj}^k \right]  \\
&\quad  - \operatorname{Ricci}^N(d\phi(\mu), d\phi(e_k)) +\kappa_k \lambda= 0,\quad k = 1, 2, \dots, n,
\end{aligned}
\end{equation}
where 
\[
P^k_{ij} = \frac{1}{2} \left( - f^j_{ik} - f^i_{jk} + f^k_{ij} \right) \text{ for all }  i, j, k = 1, 2, \dots, n,
\]
and 
\[
\mu=(\nabla^M_{e_{n+1}} e_{n+1})^\mathcal{H}.
\]
\end{proposition}

\begin{proof}
By the definition of $\lambda$-biharmonic maps and the proof of Theorem~2.1 and Theorem~3.1 in
\cite{AO2019}, we have
\begin{equation*}
\begin{aligned}
\tau_{2,\lambda}(\phi)
&=
\tau_{2}(\phi) - \lambda \tau(\phi)
\\[4pt]
&=
\sum_{i=1}^{n+1}
\Bigl\{
\nabla^{\phi}_{e_{i}} \nabla^{\phi}_{e_{i}} \tau(\phi)
 - \nabla^{\phi}_{\nabla^{M}_{e_{i}} e_{i}} \tau(\phi)
 - R^{N}\bigl( d\phi(e_{i}), \tau(\phi) \bigr)d\phi(e_{i})
\Bigr\}
\\
&\qquad
- \lambda
\sum_{i=1}^{n+1}
\Bigl\{
\nabla^{\phi}_{e_{i}} d\phi(e_{i})
  - d\phi\bigl( \nabla^{M}_{e_{i}} e_{i} \bigr)
\Bigr\}
\\[6pt]
&=
\sum_{k=1}^n(-\sum_{i=1}^{n+1} e_i e_i(\kappa_k) + \sum_{i,j=1}^n P_{ii}^j e_j(\kappa_k) + \sum_{i=1}^n \kappa_i e_i(\kappa_k)\\
&\quad  - \sum_{i,j=1}^n \left[ 2 e_i(\kappa_j) P_{ij}^k + \kappa_j (e_i P_{ij}^k) + \kappa_j P_{ij}^l P_{il}^k - \kappa_i \kappa_j P_{ij}^k - \kappa_j P_{ii}^l P_{lj}^k \right]  \\
&\quad  - \operatorname{Ricci}^N(d\phi(\mu), d\phi(e_k)) d\phi(e_k))+\sum_{k=1}^n \kappa_k \lambda d\phi(e_k)= 0.
\end{aligned}
\end{equation*}
Thus the proposition follows.
\end{proof}

This equation  (\ref{eq:lambda}) is called the \textbf{$\lambda$-biharmonic equation}.

In this paper, we use both $\sigma_{ij}$ and $\sigma_{i,j}$ to denote the same function defined on a manifold. 
The comma is used in expressions like $\sigma_{i,j+1}$ for clarity, especially when indices involve arithmetic operations.

Furthermore, the following proposition holds for general Riemannian submersions, which allows us to simplify the $\lambda$-biharmonic equation mentioned above.

\begin{proposition}
\label{proposition2.2}
Let 
$
\phi : (M^{n+1}(c), g) \longrightarrow (N^n, h)
$
be a $\lambda$-biharmonic Riemannian submersion with an adapted frame 
$
\{e_1, e_2, \dots,e_n, e_{n+1}\}
$
and the integrability data 
$
\{f_{ij}^k, \kappa_i, \sigma_{ij} \} _{i,j,k = 1,\dots,n}.
$
Let $U$ be a nonempty open set on which $\kappa:= (\kappa_1,\dots,\kappa_n)^{\mathsf T}\neq0$. Then there exists
a nonempty open subset $V\subset U$ and an adapted orthonormal frame on
$V$ such that
\[
\kappa_2=\cdots=\kappa_n=0
\]
and
\[
\sigma_{i,i+j}=0 \qquad (i=1,\dots,n-2,\ j\geq2).
\]

\end{proposition}

\begin{proof} 
The proof is the same as that of Lemma 3.2 in our previous paper \cite{MS2025}.
\end{proof}


\section{Proof of the Main Theorem}
In this section, we prove the main theorem concerning $\lambda$-biharmonic Riemannian submersions from manifolds with constant sectional curvature.

By using the reduction provided by Proposition \ref{proposition2.2} and several key lemmas, we show that, except for the critical case, any $\lambda$-biharmonic Riemannian submersion must be harmonic.
\begin{lemma}\label{lem:fiber-constant}
Let $\phi:(M^{n+1}(c),g)\to (N^n,h)$ be a $\lambda$-biharmonic Riemannian submersion.
Assume that one of the following holds.
\begin{enumerate}
\item $c\ge 0$, 
\item $c<0$ and $\lambda\neq 2(n-1)c$.
\end{enumerate}
Then, for any orthonormal frame 
 $\{e_1, \dots, e_n, e_{n+1}\}$ on $M^{n+1}(c)$ adapted to the Riemannian submersion $\phi$, 
 with $e_{n+1}$ being a vertical vector field,
  all integrability data $f^k_{ij}$, $\kappa_i$, and $\sigma_{ij}$ are constant along the fibers of $\phi$, 
  that is, the following equation holds for all $i,j,k = 1, 2, \dots, n$:
\[
e_{n+1}(f^k_{ij}) = e_{n+1}(\kappa_i) = e_{n+1}(\sigma_{ij}) = 0.
\]
\end{lemma}

\begin{proof} 
We work locally on a sufficiently small open subset $U \subset M^{n+1}(c)$ on which an adapted orthonormal frame
$\{e_1,\dots,e_n,e_{n+1}\}$ is defined, and we suppress the restriction to $U$ from the notation.

We first show that $e_{n+1}(f^k_{ij}) = e_{n+1}(\sigma_{ij}) = 0$.
This part holds for an arbitrary Riemannian submersion and the argument is the same as in our previous work \cite{MS2025}.

Next, we show  $e_{n+1}(\kappa_i) = 0$. 
From the definition of the integrability data and the Koszul formula, we obtain 

\begin{equation}\label{connection}
\left\{ \,
\begin{aligned}
& \nabla^M_{e_i} e_j = P^k_{ij} e_k - \sigma_{ij} e_{n+1} \quad  \text{for any }  i,j,k = 1, 2, \dots, n,\\
& \nabla^M_{e_{n+1}} e_{n+1} = \sum_{i=1}^n \kappa_i e_i,\\
& \nabla^M_ { e _ { i } } e _ { n + 1 } = \sigma _ { i j } e _ { j }\quad \text{for any }  i = 1, 2, \dots, n,\\
& \nabla^M_ { e _ { n+1 } } e _ { i } =\sigma _ { i j } e _ { j } -\kappa_i e_{n+1}\quad  \text{for any }  i = 1, 2, \dots, n
\end{aligned}
\right.
\end{equation}
on  $U$.
Using  (\ref{connection}) and the fact that $M^{n+1}(c)$ has constant curvature, we get 
\begin{equation}
\left\{ \,
\begin{aligned}
&- R^M _ { a (n + 1)c d} = e _ { a } ( \sigma _ { c d } ) +P_ { a  l } ^ { d } \sigma _ { c l}  - P_ { a c } ^ { l} \sigma _ { ld }- \kappa _ { c } \sigma _ { a d }  + \kappa _ { d } \sigma _ { a c } - \kappa _ { a } \sigma _ { cd } = 0, \\
&- R^M _ { a b a b } = e _ { a } ( P_ { b a } ^ { b } ) + P_ { b a } ^ { l} P_ { a l } ^ { b } + 3 \sigma _ { a b } ^ { 2 } - e _ { b } ( P _ { a a } ^ { b } ) -P_{aa}^l P_{bl}^b- f_ { a b } ^ { l } P_ { la } ^ { b }=-c, \\
&- R^M _ { a (n + 1 ) a (n + 1 ) } = - \sigma _ { a l } ^ { 2 } - e _ { a } ( \kappa _ { a } ) + P_ { a a } ^ { l } \kappa _ { l } + \kappa _ { a } ^ { 2 } = - c, \\
&- R^M _ { a (n + 1 )c (n + 1) } = - \sigma _ { c l } \sigma _ { a l } - e _ { a  } (\kappa _ { c } )+ P_ { a c } ^l\kappa_l+ \kappa _ { a }\kappa_c + e _ { n + 1 } ( \sigma _ { a c } ) = 0\quad (a\neq c) \\ 
\end{aligned}
\label{eq:R^M}
\right.
\end{equation}
on  $U$.
By differentiating both sides of the first equation in (\ref{eq:R^M}) by $e_{n+1}$, and substituting 

\begin{align*}
e_{ n + 1 }  e_a = -[e_a, e_{n+1}] +e_a e_{ n + 1 } =-\kappa_a e_{ n + 1 }  + e_a e_{ n + 1 },
\end{align*}
we obtain
\begin{align*}
- \sigma _ { a d } e _ { n + 1 } ( \kappa _ { c } ) + \sigma _ { a c } e _ { n + 1 } ( \kappa _ { d } ) - \sigma _ { c d } e _ { n + 1 } ( \kappa _ { a } ) = 0 \quad \text{for any }a,c,d= 1, 2, \dots, n 
\end{align*}
on  $U$.

First, we relabel the indices by replacing $ a $ with $ i $, $ c $ with $ j $, and $ d $ with $ k $. Under this assignment, the following identity holds.
\begin{equation}\label{eq:ijk}
\sigma_{ij} e_{n+1}(\kappa_k) = \sigma_{ik} e_{n+1}(\kappa_j) + \sigma_{jk} e_{n+1}(\kappa_i)
\end{equation}
on  $U$. 
Next, we relabel the indices by replacing $a$ with $j$, $c$ with $k$, and $d$ with $i$.
Since $\sigma = (\sigma_{ij})$ is skew-symmetric, the following identity holds.
\begin{equation}\label{eq:jki}
\sigma_{ij} e_{n+1}(\kappa_k) = -\sigma_{ik} e_{n+1}(\kappa_j) - \sigma_{jk} e_{n+1}(\kappa_i)
\end{equation}
on  $U$. 
By combining equations (\ref{eq:ijk}) with (\ref{eq:jki}), we obtain
\begin{equation}\label{sigma_kappa}
\sigma_{ij} e_{n+1}(\kappa_k) = 0
\end{equation}
for any $ i,j,k = 1, 2, \dots, n$
on  $U$.

To prove that $ e_{n+1}(\kappa_i) = 0 $ for all $ i $, we argue by contradiction.
Assume that this does not hold, meaning that we can choose an index $ a $ and a point $ p \in U $ such that
\[
e_{n+1}(\kappa_a)(p) \neq 0.
\]
Then, by continuity, there exists a nonempty open neighborhood $ \Omega_1 \subset U $ such that
\[
e_{n+1}(\kappa_a) \neq 0 \quad \text{at any point of } \Omega_1.
\]

If $\kappa_a=0$ on $\Omega_1$, then we have a contradiction.
Thus, we may assume that there exists a nonempty open set 
$\Omega_2 \subset \Omega_1$ such that 
$\kappa_a\neq0$ at any point of $\Omega_2$.

On this $\Omega_2$, 
\eqref{sigma_kappa} implies 
\[
\sigma_{ij} = 0.
\]
 Under this condition, equation  (\ref{eq:R^M}) becomes
\begin{equation}
\left\{
\begin{aligned}
&e _ { a } ( P_ { b a } ^ { b } ) - e _ { b } ( P_ { a a } ^ { b } ) = - c + P_ { a a } ^ { l } P _ { b l } ^ { b } + f _ { a b } ^ { l } P _ { la } ^ { b } - P_ { b a } ^ { l} P_ { a l } ^ { b },\\
&e _ { a } ( \kappa _ { c } ) = \delta _ { a c } c + P_ { a c } ^ { l } \kappa _ { l } + \kappa _ { a } \kappa_ { c }\\
\end{aligned}
\label{R^Msigma=0}
\right.
\end{equation}
on $\Omega_2$.
By substituting the second equation of (\ref{R^Msigma=0}) into the $\lambda$-biharmonic equation (\ref{eq:lambda})  and simplifying
 (the reduction follows from the same computation as in our previous paper \cite{MS2025}), 
 we obtain

\begin{equation}\label{eq:eek}
e _ { n+1 } e _ { n+1 } ( \kappa _a ) = - \kappa _ { a } ( \sum_{i=1}^n \kappa _ { i } ^ { 2 } + (2n-1) c-\lambda )
\end{equation}
on $\Omega_2$.
To derive additional usable equations, we differentiate both sides of the second and third equations in  (\ref{eq:R^M}) by $e_{ n+1}e_{ n+1 }$.
Then, by substituting $\sigma_{ij} = 0$ and $e_{ n + 1 }  e_a = -\kappa_a e_{ n + 1 }  + e_a e_{ n + 1 } $ we obtain

\begin{equation}
\left\{ \,
\begin{aligned}
&- 3 \{ e _ { n + 1 } ( \kappa _ { a } ) \} ^ { 2 } - 4 \kappa _a e _ { n + 1 } e _ { n + 1 } ( \kappa _a ) - P_ { a a } ^ { l } e _ { n + 1 } e _ { n + 1 }( \kappa _l ) + e _ { a } e _ { n + 1 } e _ { n + 1 } ( \kappa _a ) = 0, \\
&- 3e _ { n + 1 } ( \kappa_a ) e _ { n + 1 } ( \kappa _ { c} ) - 3 \kappa _a e _ { n + 1 } e _ { n + 1 } ( \kappa _ { c} ) - \kappa _c e _ { n + 1 } e _ { n + 1 } ( \kappa _ { a} )\\
&\qquad \qquad\qquad\qquad\qquad -P_{ac}^le _ { n + 1 } e _ { n + 1 }( \kappa _ { l} ) + e _ { a } e _ { n + 1 } e _ { n + 1 } ( \kappa _ { c } ) = 0\\
\end{aligned}
\label{iroiro}
\right.
\end{equation}
on $\Omega_2$.
By substituting (\ref{R^Msigma=0}) and (\ref{eq:eek}) into (\ref{iroiro}), we obtain the following equation:

\begin{equation}
\left\{ \,
\begin{aligned}
&- 3 \{ e _ { n+1 } ( \kappa _ { a } ) \} ^ { 2 } + \kappa _ { a } ^ { 2 } (L-3\lambda) - (K-\lambda) c=0, \\
&- 3 \{ e _ { n+1 } ( \kappa _ { b } ) \} ^ { 2 } + \kappa _ { b } ^ { 2 } (L-3\lambda) - (K-\lambda) c=0, \\
&3e _ { n+1 } ( \kappa _ { a } ) e _ { n+1 } ( \kappa _ { b } ) = \kappa _ { a } \kappa _ { b } (L-3\lambda)
\end{aligned}
\label{eq:eek^2}
\right.
\end{equation} 
on $\Omega_2$,
where
\[
L=  \sum_{i=1}^n \kappa _ { i } ^ { 2 } + (6n-5) c,~~K=\sum_{i=1}^n \kappa _ { i } ^ { 2 } + (2n-1) c.
\]

\textbf{Case I:}  $c= 0$.
Since 
$$
L = K = \sum_{i=1}^n \kappa _ { i } ^ { 2 } ,
$$
 equations \eqref{eq:eek} and \eqref{eq:eek^2} imply that 

\begin{equation}
\left\{ \,
\begin{aligned}
&e _ { n+1 } e _ { n+1 } ( \kappa _a ) = - \kappa _ { a } (   \sum_{i=1}^n \kappa _ { i } ^ { 2 }-\lambda ),\\
& 3 \{ e _ { n+1 } ( \kappa _ { a } ) \} ^ { 2 } = \kappa _ { a } ^ { 2 } ( \sum_{i=1}^n \kappa _ { i } ^ { 2 }-3\lambda) , \\
&3e _ { n+1 } ( \kappa _ { a } ) e _ { n+1 } ( \kappa _ { b } ) = \kappa _ { a } \kappa _ { b } ( \sum_{i=1}^n \kappa _ { i } ^ { 2 }-3\lambda)
\end{aligned}
\label{iroiroc=0}
\right.
\end{equation} 
on $\Omega_2$.
Using the second and third equations of \eqref{iroiroc=0} and assuming 
$e_{n+1}(\kappa_{a})\neq 0$ at any point of $\Omega_2$, we obtain
\begin{equation}\label{e1e2}
e_{n+1}(\kappa_{a})
= \frac{\kappa_{a}^{2}\left(\displaystyle\sum_{i=1}^n \kappa_i^{2} - 3\lambda\right)}
       {3\, e_{n+1}(\kappa_{a})},
\qquad
e_{n+1}(\kappa_{b})
= \frac{\kappa_{a}\kappa_{b}\left(\displaystyle\sum_{i=1}^n \kappa_i^{2} - 3\lambda\right)}
       {3\, e_{n+1}(\kappa_{a})}
\end{equation}
on $\Omega_2$.
Differentiating both sides of the second equation in \eqref{iroiroc=0} by $e _ { n + 1 }$, we obtain

\[
3\, e_{n+1}(\kappa_{a})\, e_{n+1}e_{n+1}(\kappa_{a})
=
\kappa_{a} e_{n+1}(\kappa_{a})(\sum_{i=1}^n \kappa _ { i } ^ { 2 } - 3\lambda)
+ \sum_{i=1}^n\kappa_{a}^{2}\kappa_{i} e_{n+1}(\kappa_{i})
\]
on $\Omega_2$.
Substituting the first equation of \eqref{iroiroc=0}, we obtain
\[
-3 \kappa_{a} e_{n+1}(\kappa_{a})(\sum_{i=1}^n \kappa _ { i } ^ { 2 }- \lambda)
=
\kappa_{a} e_{n+1}(\kappa_{a})(\sum_{i=1}^n \kappa _ { i } ^ { 2 } - 3\lambda)
+ \sum_{i=1}^n\kappa_{a}^{2}\kappa_{i} e_{n+1}(\kappa_{i})
\]
on $\Omega_2$.
Rearranging terms gives
\[
\kappa_{a} e_{n+1}(\kappa_{a})(4\sum_{i=1}^n \kappa _ { i } ^ { 2 } -6\lambda)
+
\sum_{i=1}^n\kappa_{a}^{2}\kappa_{i} e_{n+1}(\kappa_{i})
= 0
\]
on $\Omega_2$.
Substituting \eqref{e1e2} into this identity yields

$$
\kappa_{a}^{3}\left(\displaystyle\sum_{i=1}^n \kappa_i^{2} - 3\lambda\right)\left(\displaystyle 5\sum_{i=1}^n \kappa_i^{2} - 6\lambda\right)
      =0
$$
on $\Omega_2$.
Since $\kappa_a \neq 0$ at any point of $\Omega_2$, we have
\[
\left(\sum_{i=1}^{n}\kappa_i^{2} - 3\lambda\right)
\left(5\sum_{i=1}^{n}\kappa_i^{2} - 6\lambda\right)
= 0.
\]
Assume that there exists a nonempty open set $\Lambda_1 \subset \Omega_2$ such that

\[
5\sum_{i=1}^{n}\kappa_i^{2} - 6\lambda \neq 0
\]
on $\Lambda_1$.
Then we have
\[
\sum_{i=1}^{n}\kappa_i^{2} - 3\lambda = 0
\]
on $\Lambda_{1}$.
However, the second equation in \eqref{iroiroc=0} then implies 
$e_{n+1}(\kappa_{a}) = 0$ on $\Lambda_1$, which is a contradiction.
Hence we must have
\[
5\sum_{i=1}^{n}\kappa_i^{2} - 6\lambda  =0
\]
on $\Omega_2$.
This implies
\[
\lambda = \frac{5}{6}\sum_{i=1}^{n}\kappa_i^{2} .
\]
Substituting this into the second equation of \eqref{iroiroc=0},
\[
3 \{ e _ { n+1 } ( \kappa _ { a } ) \} ^ { 2 } = \kappa _ { a } ^ { 2 } ( \sum_{i=1}^n \kappa _ { i } ^ { 2 }-3\lambda),
\]
we obtain
\[
3\{e_{n+1}(\kappa_{a})\}^{2}
= -\frac{3}{2}\kappa_{a}^{2}\sum_{i=1}^n \kappa _ { i } ^ { 2 }
\]
on $\Omega_2$, which is a contradiction.
This contradiction shows that such an open set $\Omega_2$ cannot exist.
Therefore, we must have $\kappa_a = 0$ on $\Omega_1$.
However, this implies that $e_{n+1}(\kappa_a) = 0$ on $\Omega_1$, 
which contradicts the assumption that
$e_{n+1}(\kappa_a) \not=0$ on $\Omega_1$.
Hence, such an open set $\Omega_1$ cannot exist. 
We conclude that
\[
e_{n+1}(\kappa_i) = 0
\]
on $U$ for all $i$.

\textbf{Case II:} $c\not=0$.
Taking the difference between the first and second equations in (\ref{eq:eek^2}), we obtain
\begin{equation}
\left\{ \,
\begin{aligned}
&3 ( \{ e _ { n+1 } ( \kappa _ { a } ) \} ^ { 2 } - \{ e _ { n+1 } ( \kappa _ { b } ) \} ^ { 2 } ) = ( \kappa _ { a } ^ { 2 } - \kappa _ { b } ^ { 2 } )(L-3\lambda), \\
&3e _ { n+1 } ( \kappa _ { a } ) e _ { n+1 } ( \kappa _ { b } ) = \kappa _ { a } \kappa _ { b } (L-3\lambda)
\end{aligned}
\label{eq:eek^22}
\right.
\end{equation}
on $\Omega_2$.
Since $e_{n+1}(\kappa_a) \neq 0$ at any point of $\Omega_2$, 
we solve the second equation of (\ref{eq:eek^22}) for $e_{n+1}(\kappa_b)$, 
and substitute it into the first equation of (\ref{eq:eek^22}), which yields 
\[
\{ 3\{e _ { n+1 } ( \kappa _ { a } ) \}^2-\kappa _ { a} ^ { 2 } (L-3\lambda ) \} \{ 3\{e _ { n+1 } ( \kappa _ { a } )\}^2 +\kappa _ { b } ^ { 2 }  (L-3\lambda ) \} = 0
\]
on $\Omega_2$.
Assume that there exists a nonempty open set $\Omega_3 \subset \Omega_2$ such that
\[
3\{e _ { n+1 } ( \kappa _ { a } )\}^2 +\kappa _ { b } ^ { 2 }  (L-3\lambda )
\not=0
\]
on $\Omega_3$.
Then we have
\[
3\{e _ { n+1 } ( \kappa _ { a } ) \}^2-\kappa _ { a} ^ { 2 } (L-3\lambda ) =0
\]
on $\Omega_3$. 
Substituting this into the first equation of \eqref{eq:eek^2}, we obtain  
\[
(K-\lambda)c  =0
\]
on $\Omega_3$.
Since $c \neq 0$, we have
 \[
K=\lambda
 \]
on $\Omega_3$.
Substituting this into (\ref{eq:eek}) and the first equation of (\ref{eq:eek^2}), we obtain
\[
e _ { n+1 } e _ { n+1 } ( \kappa _a ) = 0
\]
on $\Omega_3$, and
\begin{align*}
3 \{ e _ { n+1 } ( \kappa _ { a } ) \} ^ { 2 }   &= \kappa _ { a } ^ { 2 } (L-3\lambda)  \\
 &=  \kappa _ { a } ^ { 2 } (K+(4n-4)c-3\lambda)\\
 &=  \kappa _ { a } ^ { 2 } (4(n-1)c-2\lambda) \\
 &= 2 \kappa _ { a } ^ { 2 } (2(n-1)c-\lambda) 
\end{align*}
on $\Omega_3$.

If $\lambda=2(n-1)c$, then the above identity implies that 
$$
 e _ { n+1 } ( \kappa _ { a } ) =0
 $$
 on $\Omega_3$. This contradicts our assumption.
 
Assume now that $\lambda \neq 2(n-1)c$.
Differentiating both sides of 
$$3 \{ e _ { n+1 } ( \kappa _ { a } ) \} ^ { 2 }   = 2 \kappa _ { a } ^ { 2 } (2(n-1)c-\lambda) $$
 by $e_{n+1}$ and 
combining the resulting equation with $e _ { n+1 } e _ { n+1 } ( \kappa _a ) = 0$, we obtain 
\[
\kappa_ae _ { n+1 } ( \kappa _ { a })(\lambda-2(n-1)c) =0
\]
on $\Omega_3$.
Since $\lambda \not= 2(n-1)c$, we have
\[
\kappa_ae _ { n+1 } ( \kappa _ { a })=0
\]
on $\Omega_3$.
Since $\kappa_a \neq 0$ at any point of $\Omega_3 \subset \Omega_2$ and 
$e_{n+1}(\kappa_a) \neq 0$ at any point of $\Omega_3 \subset \Omega_1$,
it follows that $\kappa_a e_{n+1}(\kappa_a) \neq 0$ on $\Omega_3$,
which contradicts the above equality.
Therefore, when $c \not=0$ we must have
\begin{equation}\label{cnot01}
3\{e _ { n+1 } ( \kappa _ { a } )\}^2 +\kappa _ { b } ^ { 2 }  (L-3\lambda )
=0
\end{equation}
on $\Omega_2$.
Since different arguments are needed depending on the dimension,
we distinguish two cases.

First, we consider the case $n \geq 3$.

Moreover, since $b(\not=a)$  is arbitrary, it follows that 
\begin{equation}\label{case2ac}
3\{e _ { n+1 } ( \kappa _ { a } )\}^2 +\kappa _ { c } ^ { 2 }  (L-3\lambda )
=0
\end{equation}
on $\Omega_2$.
Solving \eqref{cnot01} and  \eqref{case2ac} simultaneously, we obtain 
\[
\kappa_b^2=\kappa_c^2
\]
on $\Omega_2$.
From (\ref{cnot01}), we have
\[
3\{e _ { n+1 } ( \kappa _ { a } )\}^2 =-\kappa _ { b } ^ { 2 }  (L-3\lambda )
\]
on $\Omega_2$.
Since $e_{n+1}(\kappa_a) \neq 0$ at any point of $\Omega_2$, 
the left-hand side is strictly positive.
Hence, we obtain
$$
-\kappa _ { b } ^ { 2 }  (L-3\lambda )>0
$$
on $\Omega_2$.
Since $\kappa_b^2 > 0$, it follows that
$$
L-3\lambda < 0
$$
on $\Omega_2$.
Squaring the third equation of \eqref{eq:eek^2}
\[
9\{  e_{n+1}(\kappa_a)\}^2 \{ e_{n+1}(\kappa_b)\}^2
=
\kappa_a^2 \kappa_b^2 (L - 3\lambda)^2,
\]
 and substituting (\ref{cnot01}), we obtain
 \[
-\kappa_b^{2}(L-3\lambda)
3\{e_{n+1}(\kappa_b)\}^2
= \kappa_a^{2}\kappa_b^{2}(L-3\lambda)^{2}
\]
on $\Omega_2$.
Simplifying, we obtain
\[
\kappa_b^{2}(L-3\lambda)
\left(
3\{e_{n+1}(\kappa_b)\}^2
+ \kappa_a^{2}(L-3\lambda)
\right)=0
\]
on $\Omega_2$.
Assume that there exists a nonempty open set $\Lambda_2 \subset \Omega_2$ such that
\[
3\{e_{n+1}(\kappa_b)\}^2+\kappa_a^2(L-3\lambda)\not=0
\]
on $\Lambda_2$.
Then we have
\[
\kappa_b^2(L-3\lambda)=0
\]
 on $\Lambda_2$. Substituting this into (\ref{cnot01}), we obtain $e_{n+1}(\kappa_a)=0$ on $\Lambda_2 \subset \Omega_1$, which leads to a contradiction.
 Therefore,
\begin{equation}\label{case2ba}
3\{e_{n+1}(\kappa_b)\}^2+\kappa_a^2(L-3\lambda)=0
\end{equation}
on $\Omega_2$. 
Since  $L-3\lambda < 0$ and $\kappa_a\not=0$ at any point of $\Omega_2$,
it follows that
\[
e_{n+1}(\kappa_b) \not=0
\]
on $\Omega_2$.
Hence, by the same argument as in the previous discussion for $e_{n+1}(\kappa_a)$,
we obtain
\[
\{ 3\{e _ { n+1 } ( \kappa _ { b } ) \}^2-\kappa _ { b} ^ { 2 } (L-3\lambda) \} 
\{ 3\{e _ { n+1 } ( \kappa _ { b } )\}^2 +\kappa _ { c } ^ { 2 }  (L-3\lambda) \} = 0
\]
on $\Omega_2 $ for $c\not=b$.
Assume that there exists a nonempty open set $\Lambda_3 \subset \Omega_2$ such that
$$
 3\{e _ { n+1 } ( \kappa _ { b } )\}^2 +\kappa _ { c } ^ { 2 }  (L-3\lambda) \not=0.
$$
on $\Lambda_3$.
Then, by the same argument as before, we arrive at a contradiction.
Therefore, 
\[
3\{e _ { n+1 } ( \kappa _ { b } )\}^2 =-\kappa _ { c } ^ { 2 } (L-3\lambda)=-\kappa _ { b } ^ { 2 } (L-3\lambda)
\]
on $\Omega_2$.
Combining this with \eqref{case2ba}, we obtain 
\[
\kappa_a^2=\kappa_b^2
\]
on $\Omega_2$.
Differentiating both sides of this equation by $e_{n+1}$, we obtain 
\[
\kappa_ae_{n+1}(\kappa_a)=
\kappa_be_{n+1}(\kappa_b)
\]
on $\Omega_2$.
Moreover, we have 
\[
\kappa_1^2=\kappa_2^2=\cdots=\kappa_n^2\not=0
\]
on $\Omega_2$.
Then, by multiplying both sides of the third equation in \eqref{eq:eek^2} by $\kappa_a\kappa_b$, we obtain
\[
3\kappa_a e_{n+1}(\kappa_a)\kappa_be_{n+1}(\kappa_b)=\kappa_a^2\kappa_b^2(L-3\lambda)
\]
on $\Omega_2$.
Therefore, we obtain 
\[
\frac{3}{4}
\{e_{n+1}(\kappa_a^2)\}^2=\kappa_a^4(L-3\lambda)
\]
on $\Omega_2$.
Since $L-3\lambda<0$, the right-hand side is negative, whereas the left-hand side is nonnegative. 
This contradiction shows that such an open set $\Omega_2$ cannot exist.
Therefore, we must have $\kappa_a = 0$ on $\Omega_1$.
However, this implies that $e_{n+1}(\kappa_a) = 0$ on $\Omega_1$, 
which contradicts the assumption that
$e_{n+1}(\kappa_a) \not=0$ on $\Omega_1$.
Hence, such an open set $\Omega_1$ cannot exist. 
We conclude that
\[
e_{n+1}(\kappa_i) = 0
\]
on $U$ for all $i$.

Next, we consider the case $n=2$.
In this case we cannot introduce a third index $c$ as in the argument above, so a separate discussion is required.
In what follows, we may assume $a=1$ and $b=2$.
Specifically, under the condition $n=2$, equation \eqref{R^Msigma=0} becomes
\begin{equation}\label{R^Msigma=0n=2}
\left\{
\begin{aligned}
 &e_{1}(\kappa_{1}) =  \kappa_{1}^{2} - \kappa_{2} f_{1} + c, \\[4pt]
 &e_{1}(\kappa_{2}) =  \kappa_{1} f_{1} + \kappa_{1}\kappa_{2}, \\[4pt]
 &-e_{2}(f_{1}) + e_{1}(f_{2}) = f_{1}^{2} + f_{2}^{2} + c, \\[4pt]
 &e_{2}(\kappa_{1}) = - \kappa_{2} f_{2} + \kappa_{1}\kappa_{2}, \\[4pt]
 &e_{2}(\kappa_{2}) = \kappa_{1} f_{2} + \kappa_{2}^{2} + c.
\end{aligned}
\right.
\end{equation}
Furthermore, the first equation in \eqref{eq:eek^2} and  \eqref{cnot01} are reduced to the following:

\begin{equation*}
\left\{ \,
\begin{aligned}
&- 3 \{ e _ { 3 } ( \kappa _ { 1 } ) \} ^ { 2 } + \kappa _ { 1 } ^ { 2 } (L-3\lambda) - (K-\lambda) c=0, \\
&3\{e _ { 3 } ( \kappa _ { 1 } )\}^2 +\kappa _ { 2 } ^ { 2 }  (L-3\lambda )
=0.
\end{aligned}
\right.
\end{equation*} 
Eliminating $e_{3}(\kappa_{1})$ from these two equations, we obtain

\[
(\kappa_1^2+\kappa_2^2)(L-3\lambda)-(K-\lambda)c=0
\]
on $\Omega_2$.
Substituting the definitions of $L$ and $K$ yields
\[
(\kappa_1^2+\kappa_2^2)(\kappa_1^2+\kappa_2^2+7c-3\lambda)-(\kappa_1^2+\kappa_2^2+3c-\lambda)c=0.
\]
Let $Q = \kappa_1^2 + \kappa_2^2$. Then we can rewrite this as

\[
Q(Q+7c-3\lambda)-(Q+3c-\lambda)c=0
\]
on $\Omega_2$.
Simplifying, we obtain

\begin{equation} \label{Q^2}
Q^2-3(\lambda-2c)Q+(\lambda-3c)c=0.
\end{equation}
Since this is a quadratic equation in $Q$, its solutions are

\[
Q = \frac{3(\lambda - 2c) \pm \sqrt{9\lambda^{2} - 40\lambda c + 48c^{2}}}{2}
\]
on $\Omega_2$.
Thus $Q$ must be constant. Using $(\ref{R^Msigma=0n=2})$, we compute $e_1(Q) = 0$ as follows:
\[
\begin{aligned}
0 
&= e_1(Q) = e_1(\kappa_1^2 + \kappa_2^2) \\
&= 2\kappa_1\, e_1(\kappa_1) + 2\kappa_2\, e_1(\kappa_2) \\
&= 2\kappa_1(\kappa_1^2 - \kappa_2 f_1 + c)
   + 2\kappa_2(\kappa_1 f_1 + \kappa_1\kappa_2) \\
&= 2\kappa_1(\kappa_1^2 + \kappa_2^2 + c) \\
&= 2\kappa_1 (Q + c).
\end{aligned}
\]
Since $\kappa_{1}\neq 0$ at any point of $\Omega_2$, we obtain $Q + c = 0$ on $\Omega_2$, that is,
\[
\kappa_1^2+\kappa_2^2= -c.
\]

If $c>0$, then the right-hand side is negative, 
while the left-hand side is nonnegative.
This is a contradiction.

If $c<0$ and $\lambda \neq 2c$, substituting this into \eqref{Q^2}, we obtain

\[
\begin{aligned}
0
&= (-c)^2 - 3(\lambda - 2c)(-c) + (\lambda - 3c)c \\
&= c^2 + 3\lambda c - 6c^2 + \lambda c - 3c^2 \\
&= 4c(\lambda - 2c)
\end{aligned}
\]
on $\Omega_2$.
Since $c\neq 0$, it follows that 
$$\lambda = 2c$$
on $\Omega_2$.
This contradicts the assumption $\lambda \neq 2c$.
This contradiction shows that such an open set $\Omega_2$ cannot exist.
Thus, by the same argument as in the case $n \geq 3$, we conclude that
$$
e_{3}(\kappa_1)=e_{3}(\kappa_2)=0.
$$
\end{proof}

\begin{lemma}

Let 
$
\phi : (M^{n+1}(c), g) \longrightarrow (N^n, h)
$
be a $\lambda$-biharmonic Riemannian submersion with an adapted frame 
$
\{e_1, e_2,  \dots ,e_n,e_{n+1}\},
$
 the integrability data 
$
\{f_{ij}^k, \kappa_i, \sigma_{ij}\}_{i,j,k = 1,\dots,n},
$
with $\kappa_2 = \kappa_3 = \dots = \kappa_n =0$, and for any $i=1,\cdots ,n-2$, $\sigma_{i,i+j}=0,~(j\geq2)$. 
Then, the following equations hold. 
\[
\left\{
\begin{array}{l}
\displaystyle
 -\sum_{i=1}^{n} e_i e_i(\kappa_1) + \sum_{i=1}^{n} \sum_{j=1}^{n} P_{ii}^j e_j(\kappa_1) + \kappa_1 e_1(\kappa_1)+ \sum_{i=1}^{n} \sum_{l=1}^{n} \kappa_1 (P_{i1}^l)^2-3\kappa_{1} \sigma_{12}^{2} \\
 \quad - ( n - 1 ) \kappa _ { 1 } c + \kappa_{1}\lambda= 0 \quad (k=1), \\[2mm]

\displaystyle
\sum_{i=1}^{n} \left( 
2 e_i(\kappa_1) P_{i1}^3 + 
\kappa_1 e_i(P_{i1}^3) + 
\kappa_1 P_{i1}^l P_{il}^3 - 
\kappa_i \kappa_1 P_{i1}^3 - 
\kappa_1 P_{ii}^l P_{l1}^3 
\right) -3  \kappa_1 \sigma_{12} \sigma_{23}= 0\quad (k =3), \\[2mm]

\displaystyle
\sum_{i=1}^{n} \left( 
2 e_i(\kappa_1) P_{i1}^k + 
\kappa_1 e_i(P_{i1}^k) + 
\kappa_1 P_{i1}^l P_{il}^k - 
\kappa_i \kappa_1 P_{i1}^k - 
\kappa_1 P_{ii}^l P_{l1}^k 
\right) = 0\quad (k \neq 1,3).
\end{array}
\right.
\]

\end{lemma}
\begin{proof}
Using the assumptions on the adapted frame and substituting
$\kappa_2 = \kappa_3 = \dots = \kappa_n =0$ 
and for any $i=1,\cdots ,n-2$, $\sigma_{i,i+j}=0,~(j\geq2)$ into (\ref{eq:lambda}),
we obtain the above system by a direct computation.
Similar calculations can be found in the proof of Lemma~3.3 in \cite{MS2025}.
\end{proof}

\begin{theorem}[Theorem~\ref{thA} (restated)]
For $n\geq2$, 
let $\phi : (M^{n+1}(c), g) \longrightarrow (N^n, h)$ be a Riemannian submersion from a Riemannian manifold with constant sectional curvature $c$. Assume that one of the following holds.
\begin{enumerate}
    \item $c\geq 0$,  
    \item $c<0$ and $\lambda \neq 2(n-1)c$.
\end{enumerate}

Then $\phi$ is $\lambda$-biharmonic if and only if it is harmonic.
\end{theorem}

\begin{proof} 

By using the connection relations \eqref{connection} 
and the fact that $e_1, \dots, e_n$ are horizontal while $e_{n+1}$ is vertical (specifically $d\phi(e_{n+1})=0$), 
the tension field of the Riemannian submersion $\phi$ is calculated as
\begin{align*}
\tau(\phi)& = \nabla ^\phi_{e_i} d\phi(e_i) - d\phi(\nabla^M_{e_i} e_i)\\
&= -d\phi(\nabla^M_{e_{n+1}} e_{n+1})\\
&= -d\phi(\sum_{i=1}^n \kappa_i e_i)\\
&=- \sum_{i=1}^n \kappa_i \varepsilon_i.
\end{align*}

Since every harmonic map is $\lambda$-biharmonic, it remains to prove the converse under the assumptions of the theorem.
Assume that $\phi$ is $\lambda$-biharmonic. We prove that $\phi$ is harmonic.

Suppose, to the contrary, that $\phi$ is not harmonic. 
Then there exists a point $p\in M$ such that $\tau(\phi)(p)\neq 0$.
Hence, by continuity, there exists a nonempty open neighborhood $\Omega_1$ of $p$ such that
\[
\kappa=(\kappa_1,\kappa_2,\dots,\kappa_n)\neq 0
\]
on $\Omega_1$.
By Proposition~\ref{proposition2.2},
after replacing $\Omega_1$ by a smaller nonempty open subset if necessary, we may
choose an adapted orthonormal frame \(\{e_1, e_2,\dots, e_n,e_{n+1}\}\) on $\Omega_1$ such that
$$
\kappa_2=\cdots=\kappa_n=0
$$
and
$$
\sigma_{i,i+j}=0\ (i=1,\dots,n-2,\ j\geq2)
$$
on $\Omega_1$.
Consequently, we have
\[
\tau(\phi) = -\kappa_1 \varepsilon_1
\]
on $\Omega_1$.
With respect to this adapted frame on $\Omega_1$, the curvature equation \eqref{eq:R^M}
reduces to

\begin{equation}
\left\{ \,
\begin{aligned}
& R_{i(n+1) 1(n+1)}^{M} :\quad e_{i}\left(\kappa_{1}\right)=-\sigma_{12} \sigma_{i 2}+\delta_{1 i}\left(\kappa_{1}^{2}+c\right),\\
& R_{1(n+1) 1(n+1)}^{M} :\quad e_{1}\left(\kappa_{1}\right)=-\sigma_{12}^{2}+\kappa_{1}^{2}+c,\\
& R_{i(n+1) 12}^{M} :\quad e_{i}\left(\sigma_{12}\right)=P_{i 1}^{l} \sigma_{l 2}+\kappa_{1} \sigma_{i 2}+\delta_{i 1} \kappa_{1} \sigma_{12},\\
&R_{i(n+1) m(n+1)}^{M} :\quad P_{i m}^{1}=\frac{1}{\kappa_{1}}\left(\sigma_{ml}  \sigma_{i l}-\delta_{m i} c\right)\quad(m \neq 1),\\
& R_{i(n+1) j 1}^M :\quad P_{i j}^{2} \sigma_{12}=P_{i l}^{1} \sigma_{l j}-\kappa_{1} \sigma_{i j}\quad(j \geq 3),\\
& R_{1(n+1) 2 m}^{M} :\quad e_{1}\left(\sigma_{2 m}\right)=P_{12}^{l} \sigma_{l m}-P_{1 m}^{l} \sigma_{l 2}-\kappa_{m} \sigma_{12}+\kappa_{1} \sigma_{2 m},\\
&R _ { 1 (n + 1)im}^M :\qquad e _ { 1 } ( \sigma _ { im} ) =-P_ { 1m} ^ { l}  \sigma _ { li}+P_ { 1i } ^ { l} \sigma _ { lm } -\kappa _ { m}\sigma _ { 1i } + \kappa _ { 1 } \sigma _ { im} +\kappa _ { i} \sigma _ { 1m },\\ 
&R _ { 1 (n + 1)im}^M :\qquad e _ { 1 } ( \sigma _ { im} ) =-P_ { 1m} ^ { l}  \sigma _ { li}+P_ { 1i } ^ { l} \sigma _ { lm } + \kappa _ { 1 } \sigma _ { im} \quad(i,m\ge3)
\end{aligned}
\label{eq:R^M3}
\right.
\end{equation} 
on $\Omega_1$.
Consequently, we have
\begin{equation}\label{de1}
\sum_{i,m=1}^n e _ { 1 } ( \sigma _ { i m } ^ { 2 } ) =4 \kappa_{1} \sigma_{12}^{2}+2 \kappa_{1}\sum_{i,m=1}^n \sigma_{i m}^{2}, 
\quad
 e_1(\sigma _{12}^2)
=2\frac{\sigma_{12}^2}{\kappa_1}\Big(2\kappa_1^2-\sum_{l=2}^n\sigma_{l2}^2\Big)
\end{equation}
on $\Omega_1$.
Substituting \eqref{eq:R^M3} 
into the $\lambda$-biharmonic equation 
\[
 -\sum_{i=1}^{n} e_i e_i(\kappa_1) + \sum_{i=1}^{n} \sum_{j=1}^{n} P_{ii}^j e_j(\kappa_1) + \kappa_1 e_1(\kappa_1)+ \sum _ { i = 1 } ^ { n } \kappa _ { 1 } ( P_ { i 1 } ^ { l } ) ^ { 2 } -3\kappa_{1} \sigma_{12}^{2} 
- ( n - 1 ) \kappa _ { 1 } c + \kappa_{1}\lambda= 0
 \]
we obtain
\[
\kappa_1 ( -\sigma_{12}^2 + 2\sum_{i, m=1}^{n} \sigma_{i m}^{2}- \kappa_{1}^{2}-(2 n-1)c +\lambda)=0
\]
on $\Omega_1$.
The detailed computations leading to this result are analogous to those in the proof of Theorem 3.4 in \cite{MS2025}.
Since $\kappa_{1}\neq 0$ at any point of $\Omega_1$, we obtain
\begin{equation}\label{1}
 -\sigma_{12}^2 + 2\sum_{i, m=1}^{n} \sigma_{i m}^{2}- \kappa_{1}^{2}-(2 n-1)c +\lambda=0
\end{equation}
on $\Omega_1$.
 From this point, we divide the argument into the cases $n=2$ and $n\ge3$.

When $n=2$, equation \eqref{1} reduces to

$$3\sigma_{12}^2  - \kappa_1^2 - 3c + \lambda = 0.$$
Differentiating \eqref{1} by $e_{1}$ and using \eqref{de1} and \eqref{eq:R^M3}, we obtain
\begin{equation}\label{2}
7\sigma_{12}^2-\kappa_1^2-c=0
\end{equation}
on $\Omega_1$.
Repeating the same argument gives
\begin{equation}\label{3}
\kappa_1^2=15\sigma_{12}^2-c
\end{equation}
on $\Omega_1$.
From \eqref{2} and \eqref{3}, we deduce
\begin{equation}\label{4}
\sigma_{12}^2=0,\kappa_1^2=-c
\end{equation}
on $\Omega_1$.
If $c \geq 0$, this is a contradiction.
If $c < 0$ and $\lambda \neq 2c$, then, by substituting \eqref{4} into \eqref{1}, we have
$
\lambda - 2c = 0,
$
which contradicts the assumption $\lambda \neq 2c$.
Therefore, $\phi$ is harmonic when $n=2$.

We next assume $n\geq3$ and proceed similarly.
By differentiating both sides of \eqref{1} by $e_1$ and substituting \eqref{1} and (\ref{de1}), we have
\begin{equation}\label{n=2}
4\sigma_{12}^2+\frac{\sigma_{12}^2}{\kappa_1^2}\sum_{l=2}^{n} \sigma_{l2}^{2}+2( n-1)c-\lambda=0
\end{equation}
on $\Omega_1$.
By the proposition \ref{proposition2.2}, 
the second term in \eqref{n=2} reduces to the single contribution $l=3$. 
Therefore, on $\Omega_1$,
\begin{equation}\label{e1}
4\sigma_{12}^2+\frac{1}{\kappa_1^2}\sigma_{12}^2\sigma_{23}^{2}
=\lambda-2( n-1)c.
\end{equation}

\textbf{Case I:} $\lambda=2(n-1)c$. 
We first consider cases (d) and (e) in TABLE~\ref{class}.
Assume that $\lambda=2(n-1)c$, then we have $\sigma_{12}\equiv0$ on $\Omega_1$.
For $j=2,3,\cdots ,n-1$, by the first equation of \eqref{eq:R^M}, 
\[
0=e_{j}(\sigma_{1,j+1})=P_{j1}^l\sigma_{l,j+1}+\kappa_1\sigma_{j,j+1}
\]
on $\Omega_1$.
By the fourth equation of \eqref{eq:R^M3}, we have
\[
\sigma_{j,j+1}
\left(
-\sigma_{j-1,j}^2-\sigma_{j,j+1}^2-\sigma_{j+1,j+2}^2+\kappa_1^2+c
\right)
=0~~\text{for}~n \geq 4~\text{and}~ j=2,3,4,\cdots, n-2,
\]
and 
\[
\sigma_{n-1,n}
\left(
-\sigma_{n-2,n-1}^2-\sigma_{n-1,n}^2+\kappa_1^2+c
\right)
=0
\]
on $\Omega_1$.
Assume that there exists a nonempty open set
$\Omega_2\subset\Omega_1$ such that
$\sigma_{23}\neq0$
at any point of $\Omega_2$.
Then on $\Omega_2$, we have
\[
\kappa_1^2=\sigma_{23}^2+\sigma_{34}^2-c.
\]
Substituting this into \eqref{1}, which also holds on $\Omega_2$, yields a contradiction.
Thus $\sigma_{23}=0$ on $\Omega_1$.
We iterate the same argument. 
Assume that there exists a nonempty open set
$\Omega_j\subset\Omega_1$ such that 
$\sigma_{j,j+1}\neq0$
at any point of $\Omega_j$.
Then on $\Omega_j$, we have
\[
\kappa_1^2=\sigma_{j,j+1}^2+\sigma_{j+1,j+2}^2-c.
\]
Substituting this into \eqref{1} on $\Omega_j$ again gives a contradiction.
Therefore $\sigma_{j,j+1}=0$ on $\Omega_1$ for all $j=2,\dots,n-2$.
 Finally we have
 \[
 \sigma_{n-1,n}
 \left(
-\sigma_{n-1,n}^2+\kappa_1^2+c
 \right)=0
 \]
on $\Omega_1$.
Assume that there exists a nonempty open set
$\Omega_{n-1}\subset\Omega_1$ such that
$\sigma_{n-1,n}\neq0$
at any point of $\Omega_{n-1}$.
Then on $\Omega_{n-1}$, we have
\[
\kappa_1^2=\sigma_{n-1,n}^2-c
\]
on $\Omega_1$.
Substituting this into \eqref{1}, we have a contradiction.
Hence $\sigma_{n-1,n}=0$ on $\Omega_1$.
Therefore, $\sigma_{ij}=0$ for all $i,j$. By \eqref{1}, we have 
\[
\kappa_{1}^2=-c
\] 
on $\Omega_1$.
If $c \ge 0$, we obtain a contradiction. Therefore, $\phi$ is harmonic in cases (d) and (e) in TABLE~\ref{class}.


\textbf{Case II:}  $\lambda \not=2(n-1)c$.
We next consider cases (g), (h), and (i) in TABLE~\ref{class}.
From equation~\eqref{e1},
\[
4\sigma_{12}^2+\frac{\sigma_{12}^2}{\kappa_1^2}\sum_{l=2}^{n} \sigma_{l2}^{2}
=\lambda-2( n-1)c.
\]
If $\lambda - 2(n-1)c < 0$, we obtain a contradiction.
Assume next that $\lambda - 2(n-1)c > 0$.

Here and in what follows, we adopt the convention that 
$\sigma_{ij}=0$ whenever an index exceeds $n$. 
In particular, when $n=3$, we have $\sigma_{34}=\sigma_{45}=\cdots=\sigma_{n-1,n}=0$.

Differentiating both sides of \eqref{1} by $e_3$, we obtain
\[
\sigma_{12}\sigma_{23}
\bigl(
\sigma_{23}^2 + \sigma_{34}^2 + 4\kappa_1^2 - c
\bigr)
= 0 
\]
on $\Omega_1$.
Assume that there exists a nonempty open set
 $\Lambda \subset \Omega_{1}$ such that
$
\sigma_{23}^2+\sigma_{34}^2+4\kappa_1^2-c\not=0
$
at any point of  $\Lambda$.
Then we have 
$
\sigma_{12}\sigma_{23}=0
$
 on $\Lambda$.
Since $\sigma_{12}\neq 0$ by \eqref{e1}, it follows that $\sigma_{23}=0$.
Substituting this into~\eqref{e1}, we obtain
$
4\sigma_{12}^2 = \lambda - 2(n-1)c > 0 .
$
Hence $\sigma_{12}$ is a nonzero constant.
Differentiating $\sigma_{12}$ by $e_1$ and using \eqref{eq:R^M3}, we obtain
$
 e_1(\sigma_{12})=2\kappa_1 \sigma_{12} = 0 ,
$
which yields $\kappa_1=0$, a contradiction  to the fact that
$ \kappa_1 \neq 0 $ on $ \Lambda \subset \Omega_1$.
Therefore,
\begin{equation}\label{bdf}
\sigma_{23}^2 + \sigma_{34}^2 + 4\kappa_1^2 - c = 0 
\end{equation}
on $\Omega_1$.
Consequently, if $c\le 0$, we arrive at a contradiction.
Therefore, $\phi$ is harmonic in cases (h) and (i) in TABLE~\ref{class}.

Assume next that $c > 0$.
Substituting \eqref{bdf} into \eqref{1}, we obtain
\begin{equation}\label{n=5}
3\sigma_{12}^2 - 17\kappa_1^2  +(-2n+5)c + \lambda +  2\sum_{i,m=4}^{n}\sigma_{im}^{2}= 0
\end{equation}
on $\Omega_1$.
Differentiating \eqref{n=5} by $e_1$ and using \eqref{eq:R^M3}, \eqref{de1}, and \eqref{e1} we obtain
\begin{equation}\label{n=5*}
35\sigma_{12}^2 - 17\kappa_1^2 + (6n-23)c - 3\lambda  +2\sum_{i,m=4}^{n}\sigma_{im}^{2} +\frac{4\sigma_{34}^2\sigma_{45}^2}{\kappa_1^2}= 0
\end{equation}
 on $\Omega_1$.
Subtracting \eqref{n=5} from \eqref{n=5*}, we obtain
\begin{equation}\label{n=5**}
8\sigma_{12}^2 + (2n-7)c - \lambda + \frac{\sigma_{34}^2\sigma_{45}^2}{\kappa_1^2} = 0
\end{equation}
on $\Omega_1$.
Differentiating \eqref{n=5**} by $e_1$ and using \eqref{eq:R^M}, \eqref{eq:R^M3}, \eqref{de1}, and \eqref{e1}, we obtain
\begin{equation}\label{n=5***}
8(2\sigma_{12}^2 \kappa_1^2 - \sigma_{12}^2\sigma_{23}^2) +  \frac{\sigma_{34}^2\sigma_{45}^2}{\kappa_1^2}(\sigma_{12}^2-3\kappa_1^2-\sigma_{45}^2-\sigma_{56}^2) = 0
\end{equation}
on $\Omega_1$.
Set 
\[
l = \frac{\lambda - 2(n-1)c}{c} > 0.
\]
Substituting \eqref{e1}, \eqref{n=5}, and \eqref{n=5**} into \eqref{n=5***},
we obtain 
\[
8\kappa_1^2(-lc+6\sigma_{12}^2)
+(-8\sigma_{12}^2+5c+lc)(\sigma_{12}^2-3\kappa_1^2-\sigma_{45}^2-\sigma_{56}^2)
=0
\]
on $\Omega_1$.
Set
\[
a = \frac{\sigma_{12}^2}{c}, \quad 
b = \frac{\sigma_{23}^2}{c}, \quad
d = \frac{\sigma_{34}^2}{c}, \quad
e = \frac{\sigma_{45}^2}{c}, \quad
f = \frac{\sigma_{56}^2}{c}, \quad
k = \frac{\kappa_1^2}{c}. \quad 
\]
Then we obtain
\begin{equation}\label{(*)}
8k(-l+6a)+(-8a+5+l)(a-3k-e-f)=0
\end{equation}
on $\Omega_1$.
To derive a contradiction from $\eqref{(*)}$, we will now investigate the constraints on $a$, $k$, and $l$.
From \eqref{e1}, we have
\[
0 \leq \frac{\sigma_{12}^2 \sigma_{23}^2}{\kappa_1^2} = lc - 4\sigma_{12}^2
\]
on $\Omega_1$.
Since $c > 0$, $a>0$, this yields
\begin{equation}\label{x<z/4}
0<a \leq \frac{l}{4}
\end{equation}
on $\Omega_1$.
From \eqref{e1} and \eqref{bdf}, we have
\[
0 \leq \sigma_{34}^2
=c-4\kappa_1^2-\sigma_{23}^2
=c-\frac{\kappa_{1}^2}{\sigma_{12}^2}lc
\]
which implies
\begin{equation}\label{y<x/z}
k \leq  \frac{a}{l}
\end{equation}
on $\Omega_1$.
On the other hand, \eqref{1} can be rewritten as
\[
3\sigma_{12}^2
+4(\sigma_{23}^2
+\sigma_{34}^2)
+4\sum_{i=4}^{n-1}\sigma_{i,i+1}^{2}
-\kappa_1^2
+\lambda-(2n-1)c
=0.
\]
Substituting \eqref{bdf} into this equation yields
\begin{equation}\label{atkl}
3\sigma_{12}^2
+4\sum_{i=4}^{n-1}\sigma_{i,i+1}^{2}
-17\kappa_1^2
+\lambda-2(n-1)c+3c
=0
\end{equation}
on $\Omega_1$.
From \eqref{atkl}, we have
\[
0\geq-4\sum_{i=4}^{n-1}\sigma_{i,i+1}^{2}
=3\sigma_{12}^2
-17\kappa_1^2
+\lambda-2(n-1)c+3c,
\]
which implies
\begin{equation}\label{<y}
 \frac{3a+l+3}{17}\leq k
\end{equation}
on $\Omega_1$.
Combining \eqref{x<z/4}, \eqref{y<x/z}, and \eqref{<y}, we obtain
\[
 \frac{3a+l+3}{17}\leq k \leq \frac{a}{l} \leq \frac{1}{4}
\]
Hence,
\[
l \leq \frac{5}{4}-3a
\]
and since $a > 0$ and $l > 0$, we obtain
\begin{equation}\label{z<5/4}
0< l<\frac{5}{4}
\end{equation}
on $\Omega_1$.
Moreover,  by \eqref{x<z/4} and \eqref{z<5/4},
 \[
 -8a+5+l \geq  
 	-l+5
	>\frac{15}{4}
	     >0.
 \]
Thus,
 \begin{equation}\label{424x-61z-145<0}
-8a+5+l   >0
\end{equation}
on $\Omega_1$.
Using \eqref{424x-61z-145<0}, the left-hand side of \eqref{(*)} satisfies
\begin{align*}
0      &= 8k(-l+6a)+(-8a+5+l)(a-3k-e-f) \\
	&< 8k(-l+6a)+(-8a+5+l)(a-3k) \\
	&=(-8a+5+l)a +k(72a-11l-15) \\
\end{align*}
on $\Omega_1$.
Here, by \eqref{x<z/4} and \eqref{z<5/4},
\[
72a-11l-15
 \leq 7l-15
<-\frac{25}{4}
<0.
\]
Hence
 \begin{equation}\label{72a-11l-15<0}
72a-11l-15
<0.
\end{equation}
on $\Omega_1$.
Therefore, by \eqref{72a-11l-15<0} and \eqref{<y}, we have
\[
k(72a-11l-15) \leq  \frac{3a+l+3}{17}  (72a-11l-15).
\]
Thus,
\begin{align*}
0      &< (-8a+5+l)a +k(72a-11l-15) \\
	&\leq  (-8a+5+l)a +\frac{3a+l+3}{17}  (72a-11l-15) .
\end{align*}
Consequently,
\begin{align*}
0      &< 80a^2+256a+56al-11l^2-48l-45 \\
	&\leq 8l^2+16l-45 
	<-\frac{25}{2}<0
\end{align*}
on $\Omega_1$,
where the second inequality follows from \eqref{x<z/4}, and the third inequality follows from \eqref{z<5/4}.
This is a contradiction.
Therefore, $\phi$ is harmonic in case (g) in TABLE~\ref{class}.



\end{proof} 

\begin{example} \label{example1}
Let $H^m(c)$ and $H^n(c)$ be hyperbolic spaces with $m > n$, represented as
\begin{align*}
H^m(c) &= \{(x_1, \dots, x_{n-1}, x_n, \dots, x_{m-1}, z) \mid z > 0\}, \\
H^n(c) &= \{(x_1, \dots, x_{n-1}, z) \mid z > 0\},
\end{align*}
with the standard metrics
\begin{align*}
g &= -\frac{1}{cz^2} \left( \sum_{k=1}^{m-1} dx_k^2 + dz^2 \right), \\
h &= -\frac{1}{cz^2} \left( \sum_{i=1}^{n-1} dx_i^2 + dz^2 \right).
\end{align*}
Then, $\phi: H^m(c) \to H^n(c)$, 
\[
\phi(x_1, \dots, x_{n-1}, x_n, \dots, x_{m-1}, z) = (x_1, \dots, x_{n-1}, z)
\]
is a proper $\lambda$-biharmonic Riemannian submersion with $\lambda=2(n-1)c$.

An orthonormal basis for $(H^m(c), g)$ is given by
\begin{align*}
e_i &= \sqrt{-c}z \partial_{x_i} \quad (1 \le i \le n-1), \\
e_n &= \sqrt{-c}z \partial_z, \\
e_{n+j} &= \sqrt{-c}z \partial_{x_{n-1+j}} \quad (1 \le j \le m-n).
\end{align*}

Consequently, the tension field $\tau(\phi)$ and the bitension field $\tau_2(\phi)$ are given by
\begin{align*}
\tau(\phi) &= -(m-n) \sqrt{-c}\,d\phi(e_n) \neq 0, \\
\tau_2(\phi) &= 2(-c)^\frac{3}{2}(n-1)(m-n) d\phi(e_n).
\end{align*}
Now, consider the $\lambda$-bitension field defined by 
$\tau_{2,\lambda}(\phi) = \tau_2(\phi) - \lambda \tau(\phi)$.
 By setting $\lambda = 2(n-1)c$, we obtain
\begin{align*}
\tau_{2,\lambda}(\phi) 
&= \tau_2(\phi) - 2(n-1)c\tau(\phi) \\
&= 2(-c)^\frac{3}{2}(n-1)(m-n) d\phi(e_n) - 2(n-1)c \{ -(m-n)\sqrt{-c} d\phi(e_n) \} \\
&= 0.
\end{align*}
Thus, $\phi$ is a proper $\lambda$-biharmonic Riemannian submersion with $\lambda=2(n-1)c$.
\end{example}

\begin{remark}
When $n=1$, we have
\[
2(n-1)c=0,
\]
so that the cases (c) and (f) in Table~\ref{class}
reduce to the same case.
Moreover, it is known that the map
\[
\phi:(\mathbb H^2,\tfrac{dx^2+dy^2}{y^2})
\longrightarrow
\bigl((0,\infty),\tfrac{dy^2}{y^2}\bigr),
\qquad
\phi(x,y)=y,
\]
is a proper biharmonic Riemannian submersion (see, for example,
\cite{U2026}). This is precisely the case $m=2$ and $n=1$ of
Example~\ref{example1}.

\end{remark}


\section*{Acknowledgments}
The authors would like to thank Eric Loubeau and Cezar Oniciuc for valuable discussions and useful suggestions on an earlier version of this paper.






\bibliographystyle{amsbook}

\end{document}